\documentclass[11pt,abstracton,bibliography=totoc]{scrartcl}
\KOMAoptions{headings=normal}
\setkomafont{disposition}{\normalcolor\bfseries}

\usepackage{amsmath,amssymb,amsthm,dsfont,mathrsfs,stmaryrd,upref,xcolor}

\usepackage[width=430pt,top=80pt]{geometry}
\usepackage[active]{srcltx}
\usepackage[colorlinks,linkcolor = blue,citecolor = red,pagebackref=false]{hyperref}

\newcommand{\wsto}{\overset{\raisebox{-1ex}{\scriptsize $*$}}{\rightharpoondown}}

\newcommand{\N}{{\mathds{N}}}
\newcommand{\R}{{\mathds{R}}}

\newcommand{\eps}{\varepsilon}

\renewcommand{\rho}{\varrho}

\DeclareMathOperator\BV{BV}

\DeclareMathOperator*\esssup{ess\,sup}

\DeclareMathOperator\A{\mathcal A}

\newcommand{\edge}{\hspace{0.1em}\mbox{\LARGE$\llcorner$}\hspace{0.05em}}

\newcommand{\Amu}{\mathcal{A}_{u_o}^{(\mu)}}
\renewcommand\d{\mathrm{d}}
\newcommand\dx{{\,\d x}}

\newcommand\dt{{\,\d t}}
\newcommand\dxt{{\,\d x\d t}}

\newcommand{\dH}{\,\d\mathcal{H}^{n-1}}
 
\renewcommand\L{\mathrm{L}}
\newcommand\Ln{{{\cal L}^n}}
\newcommand\W{\mathrm{W}}

\newcommand{\Lw}[1]{L^{#1}_{w*}}

\DeclareMathOperator\Div{div}

\newcommand{\xint}[3]{\,{\setbox0=\hbox{$#1{#2#3}{\int}$}
   \vcenter{\hbox{$#2#3$}}\kern-.5\wd0}}

\renewcommand{\iint}{\int\hspace{-0.6em}\int}

\newtheorem{theorem}{Theorem}[section]
\newtheorem{definition}[theorem]{Definition}
\newtheorem{lemma}[theorem]{Lemma}

\theoremstyle{remark}
\newtheorem{remark}[theorem]{Remark}

\numberwithin{equation}{section}

\begin{document}

\title{On the definition of solution to the total variation flow}
\author{Juha Kinnunen and Christoph Scheven}
\maketitle

\abstract{We show that the notions of weak solution to the total
  variation flow based on the Anzellotti pairing and the variational
  inequality coincide under some restrictions on the boundary data.
  The key ingredient in the argument is a duality result for the total
  variation functional, which is based on an approximation of the
  total variation by area-type functionals.}
\section{Introduction}
This paper discusses the total variation flow
 \[
    \partial_tu-\Div\bigg(\dfrac{Du}{|Du|}\bigg)=0
    \quad\mbox{on $\Omega_T=\Omega\times (0,T)$,}
  \]
where $\Omega$ is a bounded domain in $\R^n$ and $T>0$.
For this nonlinear parabolic equation we refer to the monograph by Andreu, Caselles and Mazon \cite{Andreu-Caselles-Mazon:book}.
The total variation flow can be seen as the limiting case of the parabolic $p$-Laplace equation
  \[
    \partial_tu-\Div(|Du|^{p-2}Du)=0, \qquad 1<p<\infty,
  \]
 as $p\to1$.
A Sobolev space is the natural function space in the existence and regularity theories for a weak solution to the parabolic $p$-Laplace equation,
see the monograph by DiBenedetto \cite{DiBe1993}.
The corresponding function space for the total variation flow is functions of bounded variation and in that case the weak derivative of a function is a vector valued Radon measure. 
A standard definition of weak solution to the parabolic $p$-Laplace equation is based on integration by parts, 
but it is not immediately clear what is the corresponding definition of weak solution to the total variation flow.
One possibility is to apply the so-called Anzellotti pairing
\cite{Anzellotti:1984}. This approach has been applied for the total variation flow, for example,  in the monograph by Andreu, Caselles and Maz\'on \cite{Andreu-Caselles-Mazon:book}.

For the parabolic $p$-Laplace  equation, it is also possible to consider solutions to the parabolic variational inequality 
\[
\tfrac1p\iint_{\Omega_T}|\nabla u|^p\dxt
-\iint_{\Omega_T}u\partial_t\varphi\dxt
 \le\tfrac1p\iint_{\Omega_T}|\nabla u-\nabla\varphi|^p\dxt
 \]
 for every $\varphi\in C^\infty_0(\Omega_T)$. 
 The variational approach goes back to Lichnewsky and Temam \cite{LicT78}, who employed an analogous concept in the case 
of the time-dependent minimal surface equation. 
Wieser \cite{Wie87} showed that the variational approach gives the same class of weak solutions as the standard definition.
Moreover, he introduced a more general class of quasiminimizers related to parabolic problems.
The variational inequality related to the total variation flow is of the form
\[
     \int_0^T\|Du(t)\|(\Omega)\dt
     -\iint_{\Omega_T}u\partial_t\varphi\dxt
     \le\int_0^T\|D(u-\varphi)(t)\|(\Omega)\dt
   \]
 for every $\varphi\ C^\infty_0(\Omega_T)$, where the total variation $\|Du(t)\|(\Omega)$ is a Radon measure for almost every $t\in(0,T)$. 
 A distinctive feature is that the variational definition is based on
 total variation instead of weak gradient.

There are advantages in both approaches. 
For example, semigroup theory can be applied for the Anzellotti
pairing and the direct methods in the calculus of variations can be
applied in theory of parabolic variational integrals with linear
growth. 
  Initial and boundary value problems to the total variation flow
  have been studied by Andreu, Ballester, Caselles and Maz{\'o}n \cite{AndBCM01,ABCM} and 
by  Andreu, Caselles, D\'iaz and Maz\'on \cite{Andreu-Caselles-Diaz-Mazon:2002}. They have shown that a unique solution exists to the problem
\[
   \begin{cases}
       \partial_tu-\Div\bigg(\dfrac{Du}{|Du|}\bigg)=0&\mbox{in }\Omega\times(0,T),\\
       u(x,t)=f(x)&\mbox{on }\partial\Omega\times(0,T),\\
       u(x,t)=u_o(x)&\mbox{in }\Omega,
     \end{cases}
\]
where $\Omega$ is a bounded Lipschitz domain in $\R^n$, $f\in L^1(\partial\Omega)$ and $u_0\in L^1(\Omega)$.
The case of homogeneous boundary data $f=0$ is discussed in
\cite{Andreu-Caselles-Diaz-Mazon:2002} and \cite{ABCM} discusses
inhomogeneous time-independent boundary data.
The Neumann problem for the total variation flow has
  been studied in \cite{AndBCM01}.
The concepts of solution discussed in \cite{ABCM} are more general than the concept of variational solution considered in this work.
A variational approach to existence and uniqueness questions has been
discussed by B\"ogelein, Duzaar and Marcellini \cite{BDM}, see also
\cite{BoegelDuzSchev:2016}, and for the
corresponding obstacle problem by B\"ogelein, Duzaar and Scheven \cite{BoegelDuzSchev:2015}.
A necessary and sufficient condition for continuity of a variational solution has been proved by DiBenedetto, Gianazza and Klaus \cite{DiBeGiaKla2017}.
Gianazza and Klaus \cite{GiaKlaus2019} showed that variational solutions to the Cauchy-Dirichlet problem for the total variation flow are obtained as the limit as $p\to1$ of variational solutions to the corresponding problem for the parabolic $p$-Laplace equation.
See also  B\"ogelein, Duzaar, Sch\"atzler and Scheven \cite{BDSS:2019}.

Our main result in Theorem \ref{thm:main} below shows that the notions of weak solution to the total variation flow based on the Anzellotti pairing and the variational inequality coincide under natural assumptions.
We consider weak solutions to a Cauchy-Dirichlet problem for the total variation flow, which can formally be written as
  \begin{equation}
   \label{Cauchy-Dirichlet}
     \begin{cases}
       \partial_tu-\Div\bigg(\dfrac{Du}{|Du|}\bigg)=0&\mbox{in
                                           }\Omega_T,\\
       u=u_o&\mbox{on }\partial_{\mathcal{P}}\Omega_T.
     \end{cases}
 \end{equation}
 For this problem, we consider the appropriate definitions of weak solution with time-dependent boundary values,
 see Definition \ref{def:Mazon-solution} and Definition \ref{def:var-sol} below.
  It is relatively straight forward to show that a weak solution to the total variation flow is a variational solution.
 This question has been studied in the context of metric measure spaces in \cite{Gorny-Mazon:2021}.
 However, it is much more challenging to prove that a variational solution is a weak solution.
The key ingredient is a duality result for the total variation functional in Theorem \ref{thm:TV-duality}. 
This is based on an approximation of the total variation by area-type functionals, see Theorem \ref{thm:area-duality}.

\subsubsection*{Acknowledgements} The first author would like to thank Heikki Hakkarainen, Panu Lahti and Olli Saari for several 
useful discussions on this topic over the years.

\section{Preliminaries}

\subsection{Functions of bounded variation}

Throughout this article, we consider a bounded Lipschitz domain
$\Omega\subset\R^n$. 
We will
prescribe Dirichlet boundary values on $\partial\Omega$ in form of
\emph{solid boundary values}. In the stationary case, this means that we choose 
an open Lipschitz domain $\Omega^\ast\subset\R^n$ with
$\Omega\Subset\Omega^\ast$, consider the Dirichlet data 
\begin{equation}
  \label{bdry-data-assum}
  u_o\in \W^{1,1}(\Omega^\ast)\cap\L^2(\Omega^\ast),
\end{equation}
and restrict ourselves to functions that agree with $u_o$
almost everywhere with respect to the Lebesgue
measure $\Ln$ on $\Omega^\ast\setminus\Omega$.  
  We point out that in the parabolic setting,
  we will consider time-dependent boundary
data as in \eqref{bdry-data-parabolic}, which satisfies
\eqref{bdry-data-assum} on almost every time slice.
The space $\BV(\Omega^\ast)$ is defined as the space of functions
$u\in L^1(\Omega^\ast)$ for which the distributional derivative $Du$ is
given by a finite vector-valued Radon measure on $\Omega^\ast$. By $\|Du\|$ we
denote the total variation measure of $Du$, which is defined by
\begin{equation*}
  \|Du\|(A):=\sup\bigg\{\sum_{i=1}^\infty |Du(A_i)|\colon
  A_i \mbox{ are pairwise disjoint Borel sets with }
  A=\bigcup_{i=1}^\infty A_i\bigg\}
\end{equation*}
for every Borel set $A\subset\Omega^\ast$, cf. \cite[Def.~1.4]{AmbrosioFuscoPallara}.
For the Dirichlet problem, we consider the function class
\begin{equation*}
  \BV_{u_o}(\Omega):=\{u\in \BV(\Omega^\ast)\colon u=u_o \mbox{
    a.e. in }\Omega^\ast\setminus\Omega\}.
\end{equation*}
For $u\in\BV_{u_o}(\Omega)$, we write $D^au$ and $D^su$ for the
absolutely continuous and the singular part of $Du$ with respect to the Lebesgue
measure $\Ln$ and, moreover, $\nabla u$ for the Radon-Nikodym derivative of
$Du$ with respect to the Lebesgue measure. With this notation, we have
the decomposition
\begin{equation}\label{decomposition-Du}
  Du=D^au+D^su=\nabla u\Ln+D^su
\end{equation}
for every $u\in\BV_{u_o}(\Omega)$. From \cite[Thm. 3.88]{AmbrosioFuscoPallara} we know that on a bounded
Lipschitz domain $\Omega\subset\R^n$,  
there exist bounded inner and outer trace operators
\begin{equation*}
  T_\Omega: \BV(\Omega)\to L^1(\partial\Omega)
  \qquad\mbox{and}\qquad
  T_{\R^n\setminus\overline\Omega}: \BV(\R^n\setminus\overline\Omega)\to L^1(\partial\Omega).
\end{equation*}
With these trace operators, we have the following extension result for
$\BV$-functions. 
\begin{lemma}[{\cite[Cor. 3.89]{AmbrosioFuscoPallara}}]
\label{lem:ambrosio-et-al}
  Let $\Omega\subset\R^n$ be an open set with bounded Lipschitz
  boundary, $u\in \BV(\Omega)$ and
  $v\in\BV(\R^n\setminus\overline\Omega)$. Then the function
  \begin{equation*}
    w(x)=
    \begin{cases}
      u(x),& \mbox{for }x\in\Omega,\\
      v(x),& \mbox{for }x\in\R^n\setminus\overline\Omega,
    \end{cases}
  \end{equation*}
  belongs to $\BV(\R^n)$, and its derivative is given by the measure
  \begin{equation*}
    Dw= Du+Dv+\big(T_{\R^n\setminus\overline\Omega}v-T_\Omega
    u\big)\nu_\Omega\mathcal{H}^{n-1}\edge
    \partial\Omega,
  \end{equation*}
  where $\nu_\Omega$ denotes the generalized outer unit normal to
  $\Omega$. In the above formula, we interpret $Du$ and $Dv$ as
 vector-valued  measures on the entire $\R^n$ that are concentrated in
 $\Omega$ and in $\R^n\setminus\overline\Omega$, 
  respectively.
\end{lemma}

We apply this lemma with the boundary values $u_o$ as in \eqref{bdry-data-assum} in place of
$v$. This is possible because we can extend the boundary values to a
function $u_o\in W^{1,1}(\R^n)$ without changing the boundary condition. 

\subsection{Parabolic function spaces}\label{sec:parabolic-spaces}\label{sec:spaces}

A map $v\colon [0,T]\to X$ into a Banach space $X$ is
called \emph{Bochner measurable} or \emph{strongly measurable}
if it can be approximated by simple functions
$v_k\colon [0,T]\to X$ in the sense $\|v_k(t)-v(t)\|_X\to0$ for
a.e. $t\in [0,T]$ as $k\to\infty$. 
A simple function is of the form 
\[
v_k(t)=\sum_{i=1}^N v^{(i)}\chi_{E_i}(t)
\] 
for
$v^{(1)},\ldots,v^{(N)}\in X$ and pairwise disjoint
measurable sets $E_1,\ldots,E_N\subset[0,T]$.
For $1\le p\le\infty$, 
we write $L^p(0,T;X)$ for the space of equivalence classes of
Bochner measurable functions
$v\colon [0,T]\to X$ with $\|v(t)\|_X\in L^p([0,T])$. 

For non-separable Banach spaces, the assumption of Bochner measurability
often turns out to be too strong. For maps into non-separable dual spaces $X=X_0'$,
we use the following weaker condition. 
A function $v\colon [0,T]\to X_0'$ 
is called weakly$\ast$-measurable if
the map $[0,T]\ni t\mapsto \langle v(t),\varphi\rangle\in \R$ 
is measurable for every $\varphi\in X_0$, where 
$\langle\cdot,\cdot\rangle$ denotes the dual pairing between
$X_0'$ and $X_0$.
Using this concept of measurability, we
introduce the weak$\ast$-Lebesgue space
\begin{equation}\label{def-weak-Lebesgue}
  \Lw{p}(0,T;X_0'):=
  \bigg\{v\colon [0,T]\to X_0'\ \bigg|\,
  \begin{array}{l}
    v\mbox{ is weakly$*$-measurable with }\\
    t\mapsto \|v(t)\|_{X_0'}\in L^p([0,T])
  \end{array}
  \bigg\}\,,
\end{equation}
with $1\le p\le\infty$.
The weak$\ast$-Lebesgue space with exponent $p=\infty$ naturally occurs in the case of the dual space
$W^{-1,\infty}(\Omega)=[W^{1,1}_0(\Omega)]'$, since 
\begin{equation}\label{parabolic-dual-space}
  \big[L^1(0,T;W^{1,1}_0(\Omega))\big]'
  =
  \Lw\infty(0,T;W^{-1,\infty}(\Omega)),
\end{equation}
cf. \cite[Sect. VII.4]{Theory-of-lifting}.
Moreover, from \cite[Remark 3.12]{AmbrosioFuscoPallara} we know that  
$\BV(\Omega^\ast)$ is a dual space with a separable pre-dual $X_0$,
whose elements take the form $g-\Div G$ with $g\in C^0_0(\Omega^\ast)$ and
$G\in C^0_0(\Omega^\ast,\R^n)$.
  We will consider the Cauchy-Dirichlet problem for the 
  total variation flow with time-dependent boundary data satisfying
  \begin{equation}
    \label{bdry-data-parabolic}
    u_o\in L^1(0,T;W^{1,1}(\Omega))\cap C^0([0,T];L^2(\Omega)).
  \end{equation}
The natural solution space for this problem 
is the weak$\ast$-Lebesgue space
\begin{equation*}
  \Lw1(0,T;\BV_{u_o}(\Omega))
  :=
  \big\{ u\in \Lw1(0,T;\BV(\Omega^\ast))\,\colon
  u=u_o\mbox{ a.e. in $(\Omega^\ast\setminus\Omega)_T$}
  \big\}.
\end{equation*}
The next lemma states that weak$\ast$-measurability implies measurability of the total variation functional. 
\begin{lemma}\label{lem:TV-measurable}
  For $u\in\Lw1(0,T;\BV_{u_o}(\Omega))$, the total variation 
  $\|Du(t)\|(\overline\Omega)$ depends measurably on $t\in(0,T)$. 
\end{lemma}

\begin{proof}
The total variation of the open set $\Omega^\ast$ is given by 
\begin{equation}\label{def:tot-var}
	\| Du(t)\| (\Omega^\ast)
	:=
	\sup\bigg\{\int_{\Omega^\ast} u(t) \Div\zeta\dx\  \bigg|\ \zeta\in C^1_0(\Omega^\ast, \R^n),\,
	\|\zeta\|_{L^\infty(\Omega^\ast)}\le 1\bigg\},
\end{equation}
cf. \cite[Prop. 3.6]{AmbrosioFuscoPallara}.      
The integrals in the supremum depend measurably on time by definition
of the weak$\ast$-measurability. Since the supremum is taken over a
separable set, the supremum is measurable as well. 
Therefore,
\begin{equation*}
  \|Du(t)\|(\overline\Omega)
  =
  \|Du(t)\|(\Omega^\ast)
  -
  \int_{\Omega^\ast\setminus\Omega}|\nabla u_o(t)|\dx
\end{equation*}
depends measurably on $t\in(0,T)$. 
\end{proof}

\subsection{The area functional}
For a parameter $\mu\ge0$, we consider the area functional
\begin{equation}\label{def:areafcnl}
  \mathcal A^{(\mu)}_{u_o}(u):=\int_{\Omega}\sqrt{\mu^2+|\nabla u|^2}\dx
  +\|D^su\|(\overline\Omega)
\end{equation}
where $u\in\BV_{u_o}(\Omega)$.
The limit case $\mu=0$ corresponds to the total variation functional,
i.e. 
\[
\mathcal A_{u_o}^{(0)}(u)=\|Du\|(\overline\Omega)
\] 
for every $u\in\BV_{u_o}(\Omega)$.
We point out that the functional $\mathcal A^{(\mu)}_{u_o}$ depends on
the prescribed boundary function $u_o$. More precisely, since $\Omega$ is a
bounded Lipschitz domain, Lemma~\ref{lem:ambrosio-et-al} gives the
decomposition 
\begin{equation*}
  Du= Du\edge\Omega+Du_o\edge(\Omega^\ast\setminus\overline\Omega)
  +\big(T_{\R^n\setminus\overline\Omega}u_o-T_\Omega u\big)\nu_\Omega\mathcal{H}^{n-1}\edge
  \partial\Omega 
\end{equation*}
for every $u\in\BV_{u_o}(\Omega)$.
Therefore, the last term in~\eqref{def:areafcnl} can be expressed as 
\begin{align*}
  \|D^su\|(\overline\Omega)
  &=
  \|D^su\|(\Omega)+\|D^su\|(\partial\Omega)\\
  &=
  \|D^su\|(\Omega)+\int_{\partial\Omega}|T_\Omega u-T_{\R^n\setminus\overline\Omega}u_o|\,\d \mathcal{H}^{n-1},
\end{align*}
which implies 
\begin{equation*}
  \mathcal A^{(\mu)}_{u_o}(u)=\int_{\Omega}\sqrt{\mu^2+|\nabla u|^2}\dx
  +\|D^su\|(\Omega)+\int_{\partial\Omega}|T_\Omega u-T_{\R^n\setminus\overline\Omega}u_o|\,\d \mathcal{H}^{n-1}.
\end{equation*}

The following approximation result for $\BV$-functions will be useful for us.

\begin{lemma}[Strict interior approximation]\label{lem:interior-approximation}
  Let $\Omega\subset\R^n$ be a bounded Lipschitz domain and 
  $u_o\in W^{1,1}(\Omega^\ast)\cap L^2(\Omega^\ast)$. For every
  $u\in\BV_{u_o}(\Omega)$ and $\mu\in[0,1]$, there exists a sequence of functions $u_i\in
  u_o+C^\infty_0(\Omega)$, $i\in\N$, with  $u_i\to u$ in $L^2(\Omega)$ and  $\Amu(u_i)\to \Amu(u)$
  as $i\to\infty$.
\end{lemma}
\begin{proof}
 With the mollification operator $M_\eps$ defined in
 \cite[Sect. 5]{BDSS:2019} we define 
 \[
 u_i=u_o+M_{\eps_i}(u-u_o), 
 \qquad i\in\N,
 \] 
 for some sequence $\eps_i\downarrow0$ as $i\to\infty$. From \cite[Lemma 5.1]{BDSS:2019}
 we infer $u_i\to u$ in $L^2(\Omega)$ and
 $\mathcal A_{u_o}^{(1)}(u_i)\to \mathcal A_{u_o}^{(1)}(u)$ as $i\to\infty$. For the other parameters $\mu\in[0,1)$, the
 asserted convergence follows from
 the Reshetnyak continuity theorem \cite[Thm. 2.39]{AmbrosioFuscoPallara}.
\end{proof}


\section{The notion of weak solution by Anzellotti pairing}

For any vector field $z\in L^\infty(\Omega,\R^n)$ with $\Div z\in
L^2(\Omega)$, there exists a uniquely determined outer normal trace $[z,\nu]\in
L^\infty(\partial\Omega)$ with
\[
\big\|[z,\nu]\big\|_{L^\infty(\partial\Omega)}\le\|z\|_{L^\infty(\Omega)}
\] 
and 
\begin{equation}
  \label{property-normal-trace}
  \int_\Omega w\Div z\dx+\int_\Omega z\cdot\nabla w\dx
  =
  \int_{\partial\Omega}[z,\nu]w\dH
\end{equation}
for every $w\in W^{1,1}(\Omega)\cap L^2(\Omega)$, see \cite[Prop. C.4]{Andreu-Caselles-Mazon:book}.
We use the normal trace for the following version of an Anzellotti pairing, which is tailored for the Dirichlet problem.
\begin{definition}
  For any $u\in\BV_{u_o}(\Omega)\cap L^2(\Omega^\ast)$ and
  $z\in L^\infty(\Omega,\R^n)$ with
  $\Div z\in L^2(\Omega)$, we define the Anzellotti pairing of
  $z$ and $Du$ as the distribution
  \begin{equation*}
    (z,Du)_{u_o}(\varphi):=-\int_{\Omega}\Div z\,u\varphi\dx
    -\int_{\Omega}z\cdot\nabla\varphi \,u\dx
    +\int_{\partial\Omega}[z,\nu] u_o\varphi\dH
  \end{equation*}
  for every $\varphi\in C^\infty_0(\Omega^\ast)$.
\end{definition}

It turns out that this distribution is a measure. 

\begin{lemma}\label{lem:anzellotti-bound}
  For any $u\in\BV_{u_o}(\Omega)\cap L^2(\Omega^\ast)$ and
  $z\in L^\infty(\Omega,\R^n)$ with
  $\Div z\in L^2(\Omega)$, the pairing $(z,Du)_{u_o}$ defines a
  Radon measure on $\overline\Omega$, and we have 
\begin{equation}\label{Anzellotti-bound}
  \big|(z,Du)_{u_o}(\overline\Omega)\big|\le \|z\|_{L^\infty(\Omega)}\|Du\|(\overline\Omega).
\end{equation}
\end{lemma}
Before giving the proof, we state a variant of the preceding
estimate that involves the area
functional instead of the total variation. To this end, we note that   
for any vectors $z,v\in\R^n$ with $|z|\le 1$ and $\mu>0$, we have the Fenchel-type inequality
\begin{equation}\label{Fenchel-pointwise}
  |z\cdot v|\le \sqrt{\mu^2+|v|^2}-\mu\sqrt{1-|z|^2}.
\end{equation}
This inequality can be verified by a straightforward calculation or by noting that
$f_\mu^*(z)=-\mu\sqrt{1-|z|^2}$ is the convex conjugate function of
$f_\mu(v)=\sqrt{\mu^2+|v|^2}$ and recalling the general Fenchel
inequality $|z\cdot v|\le f_\mu(z)+f_\mu^\ast(v)$. We note that equality
in~\eqref{Fenchel-pointwise} holds if and only if
$z=v(\mu^2+|v|^2)^{-1/2}$. The Fenchel-type estimate above leads
to an estimate for the Anzellotti pairing.

\begin{lemma}\label{lem:Fenchel-Anzellotti}
  For every $\mu\in(0,1]$, every
  $u\in\BV_{u_o}(\Omega)\cap L^2(\Omega^\ast)$ and every $z\in L^\infty(\Omega,\R^n)$
  with $\|z\|_{L^\infty}\le1$ and $\Div z\in L^2(\Omega)$ we have 
  \begin{equation}\label{Anzellotti-Fenchel-bound}
    \big|(z,Du)_{u_o}(\overline\Omega)\big|
    \le
    \A_{u_o}^{(\mu)}(u)-\mu\int_\Omega \sqrt{1-|z|^2}\dx.
  \end{equation}
\end{lemma}

\begin{proof}[Proof of Lemmas~\ref{lem:anzellotti-bound} and~\ref{lem:Fenchel-Anzellotti}]
  Let $u\in\BV_{u_o}(\Omega)\cap L^2(\Omega^\ast)$. 
  Lemma~\ref{lem:interior-approximation} provides us with a sequence
  of approximating functions
  $u_i\in u_o+C^\infty_0(\Omega)$, $i\in\N$, that converges strictly to $u$, i.e.  $u_i\to u$ in $L^2(\Omega)$ and
  \begin{equation}\label{strict-convergence}
    \|\nabla u_i\|_{L^1(\Omega)}\to \|Du\|(\overline\Omega)
  \end{equation}
 as $i\to\infty$.
  For every test function $\varphi\in
  C^\infty_0(\Omega^\ast)$ we obtain
  \begin{align}\label{approximation-anzellotti}
    (z,Du)_{u_o}(\varphi)
    &=
      -\int_{\Omega}\Div z\,u\varphi\dx
      -\int_\Omega z\cdot\nabla\varphi\, u\dx
      +\int_{\partial\Omega}[z,\nu] u_o\varphi\dH\\\nonumber
    &=
      \lim_{i\to\infty}\bigg(-\int_{\Omega}\Div z\,u_i\varphi\dx
       -\int_\Omega z\cdot\nabla\varphi\, u_i\dx\bigg)
      +\int_{\partial\Omega}[z,\nu] u_o\varphi\dH\\\nonumber
    &=
      \lim_{i\to\infty}\int_{\Omega} z\cdot\nabla u_i\,\varphi\dx,
  \end{align}
  where in the last step, we applied~\eqref{property-normal-trace} with $w=u_i\varphi$ and the fact $u_i=u_o$ on $\partial\Omega$ in the sense
  of traces. By~\eqref{strict-convergence}, we deduce
  \begin{align*}
    |(z,Du)_{u_o}(\varphi)|
    &\le
     \lim_{i\to\infty}\|\nabla
      u_i\|_{L^1(\Omega)}\|z\|_{L^\infty(\Omega)}\|\varphi\|_{C^0(\overline\Omega)}\\
      &=\|Du\|(\overline\Omega)\|z\|_{L^\infty(\Omega)}\|\varphi\|_{C^0(\overline\Omega)}.
  \end{align*}
  This implies that $(z,Du)_{u_o}$ defines a measure on
  $\overline\Omega$ that satisfies~\eqref{Anzellotti-bound}. This
  completes the proof of Lemma~\ref{lem:anzellotti-bound}. For the
  proof of Lemma~\ref{lem:Fenchel-Anzellotti}, we observe that
  according to Lemma~\ref{lem:interior-approximation}, the sequence of
  approximating functions $u_i\in u_o+C^\infty_0(\Omega)$, $i\in\N$, has the property 
   $ \Amu(u_i)\to\Amu(u)$ as $i\to\infty$ for every $\mu\in(0,1]$. 
  We use \eqref{approximation-anzellotti} with a cut-off
  function $\varphi\in C^\infty_0(\Omega^\ast)$ with $\varphi\equiv1$
  on $\Omega$.  Estimating the last integrand in
  \eqref{approximation-anzellotti} 
  by means of~\eqref{Fenchel-pointwise},
  we arrive at 
  \begin{align*}
    \big|(z,Du)_{u_o}(\overline\Omega)\big|
    &=\big|(z,Du)_{u_o}(\varphi)\big|\\
    &\le
    \lim_{i\to\infty}\int_{\Omega} \sqrt{\mu^2+|\nabla u_i|^2}\dx
    -
    \mu\int_{\Omega} \sqrt{1-|z|^2}\dx\\
    &=
    \Amu(u)
    -
    \mu\int_{\Omega} \sqrt{1-|z|^2}\dx.
  \end{align*}
  This completes the proof of Lemma~\ref{lem:Fenchel-Anzellotti}.
\end{proof}

Next, we state an elementary identity for the Anzellotti pairings that
will frequently be used in the proofs that follow. 
\begin{lemma}\label{lem:comparison-pairings}
  Let $z\in L^\infty(\Omega,\R^n)$ be with $\Div z\in
  L^2(\Omega)$ and $v\in\BV_{u_o}(\Omega)\cap L^2(\Omega)$. Then we
  have 
  \begin{equation*}
    (z,Dv)_{u_o}(\overline\Omega)
    -
    \int_\Omega z\cdot\nabla u_o\dx
    =
    \int_\Omega \Div z\, (u_o-v) \dx.
  \end{equation*}
\end{lemma}

\begin{proof}
   We use first the definition of the Anzellotti
   pairing and then property~\eqref{property-normal-trace}
   of the normal trace with
   $w=u_o$ in order to have
   \begin{align*}
     (z,Dv)_{u_o}(\overline\Omega)
     &=
       -\int_\Omega v\Div z \dx
       +\int_{\partial\Omega}[z,\nu]u_o\dH\\
     &=
       -\int_\Omega \Div z\,(v-u_o) \dx
       -\int_\Omega u_o\Div z \dx
       +\int_{\partial\Omega}[z,\nu]u_o\dH
       \\
     &=
       -\int_\Omega \Div z\,(v-u_o) \dx
       +
       \int_\Omega z\cdot\nabla u_o\dx.   \qedhere  
   \end{align*}
\end{proof}

We apply the following definition of weak solution.

\begin{definition}[Weak solution]
  \label{def:Mazon-solution}
   Assume that $u_o\in L^1(0,T;W^{1,1}(\Omega))\cap
   C^0([0,T];L^2(\Omega))$.
   We say that a function $u\in \Lw1(0,T;BV_{u_o}(\Omega))\cap
   C^0([0,T];L^2(\Omega))$ with $\partial_tu\in L^2(\Omega_T)$ is 
   a weak solution of~\eqref{Cauchy-Dirichlet}
   if $u(0)=u_o(0)$ and if there exists a vector field $z\in L^\infty(\Omega_T,\R^n)$ with
   $\|z\|_{L^\infty}\le1$, $\Div z=\partial_tu$ in $\Omega_T$
  in the sense of distributions and
  \begin{equation}\label{mazon-TV}
    \|Du(t)\|(\overline\Omega)
    +\int_{\Omega\times\{t\}}\partial_tu(u-v)\dx
    =
    (z(t),Dv)_{u_o}(\overline\Omega)
  \end{equation}
  for every $v\in\BV_{u_o(t)}(\Omega)\cap L^2(\Omega)$ and a.e. $t\in(0,T)$.
\end{definition}

We recall an equivalent way to formulate the preceding concept
of solution that has already been observed in
\cite[Thm. 1]{Andreu-Caselles-Diaz-Mazon:2002}. 

\begin{lemma}\label{lem:Mazon-2}
   A map $u\in \Lw1(0,T;BV_{u_o}(\Omega))\cap
   C^0([0,T];L^2(\Omega))$ with $\partial_tu\in L^2(\Omega_T)$ and
   $u(0)=u_o(0)$ is 
   a weak solution of~\eqref{Cauchy-Dirichlet} in the sense of Definition~\ref{def:Mazon-solution} 
   if and only if there exists a vector field
   $z\in L^\infty(\Omega_T,\R^n)$ with
   \begin{equation}\label{1-harmonic-flow}
     \Div z=\partial_tu \qquad\mbox{in $\Omega_T$}
  \end{equation}
  in the sense of distributions, for which 
  \begin{equation}\label{maximal-pairing}
    \|z\|_{L^\infty(\Omega_T)}\le 1
    \qquad\mbox{and}\qquad
    \big(z(t),Du(t)\big)_{u_o}(\overline\Omega)
    =
    \|Du(t)\|(\overline\Omega)
  \end{equation}
  hold true for a.e. $t\in(0,T)$.
\end{lemma}

\begin{remark}
  The condition~\eqref{maximal-pairing} for the vector field $z$ can
  be interpreted as an analogue of the identity $z=\frac{Du}{|Du|}$
  for $\BV$-functions. In this sense, equation~\eqref{1-harmonic-flow}
  is the generalization of the differential equation
  \eqref{Cauchy-Dirichlet}$_1$ to the $\BV$-setting.
\end{remark}

\begin{proof}[Proof of Lemma~\ref{lem:Mazon-2}]
   If  $u$ is a weak solution of~\eqref{Cauchy-Dirichlet} in the sense
   of Definition~\ref{def:Mazon-solution}, we simply choose $v=u(t)$
   in~\eqref{mazon-TV} to deduce~\eqref{maximal-pairing}.

   For the other direction, assume that $u\in \Lw1(0,T;BV_{u_o}(\Omega))\cap
   C^0([0,T];L^2(\Omega))$, with $\partial_tu\in L^2(\Omega_T)$, and
   $u(0)=u_o(0)$ 
   and that there exists a vector field
   $z\in L^\infty(\Omega_T,\R^n)$ with the properties in~\eqref{1-harmonic-flow} 
   and~\eqref{maximal-pairing}. For
   $v\in\BV_{u_o(t)}(\Omega)\cap L^2(\Omega)$, we apply
   Lemma~\ref{lem:comparison-pairings}, once with $v$ and once with
   $u(t)$, to obtain
   \begin{align*}
      \big(z(t),Dv\big)_{u_o}(\overline\Omega)
      &=
      \int_\Omega z(t)\cdot\nabla u_o(t)\dx
      +
      \int_{\Omega\times\{t\}} \partial_tu (u_o-v) \dx\\
      &=
      \big(z(t),Du(t)\big)_{u_o}(\overline\Omega)
      +
        \int_{\Omega\times\{t\}} \partial_tu (u-v) \dx\\
      &=
        \|Du(t)\|(\overline\Omega)
      +
        \int_{\Omega\times\{t\}} \partial_tu (u-v) \dx,
   \end{align*}
   for a.e. $t\in(0,T)$. 
   In the last line, we used~\eqref{maximal-pairing}. This proves that $u$ is a
   solution in the sense of Definition~\ref{def:Mazon-solution}.
\end{proof}

\section{The concept of variational solution}

The following notion of solution of the Cauchy-Dirichlet problem~\eqref{Cauchy-Dirichlet} is
based on the variational approach by Lichnewsky and Temam \cite{LicT78}.

\begin{definition}[Variational solution]\label{def:var-sol}
  Assume that the initial and
  boundary values satisfy $u_o\in L^1(0,T;W^{1,1}(\Omega))\cap
   C^0([0,T];L^2(\Omega))$.
   A function $u\in \Lw1(0,T;BV_{u_o}(\Omega))\cap  C^0([0,T];L^2(\Omega))$
   is called a variational solution of~\eqref{Cauchy-Dirichlet} if
   \begin{align}\label{var-sol-TVF}
     \int_0^T\|Du(t)\|(\overline\Omega)\dt
     &\le
       \iint_{\Omega_T}\partial_tv(v-u)\dxt
       +\int_0^T\|Dv(t)\|(\overline\Omega)\dt\\\nonumber
     &\quad
        -\tfrac12\int_{\Omega}|v(T)-u(T)|^2\dx
        +\tfrac12\int_{\Omega}|v(0)-u_o(0)|^2\dx
   \end{align}
   holds true for every $v\in\Lw1(0,T;\BV_{u_o}(\Omega))\cap
   C^0([0,T];L^2(\Omega))$ with $\partial_tv\in L^2(\Omega_T)$.
 \end{definition}

In the following, we will consider variational solutions with the
additional property $\partial_tu\in L^2(\Omega_T)$, as it is required
in the notion of weak solution in the sense of Definition~\ref{def:Mazon-solution}. In this case, the
variational inequality can also be considered separately on the time
slices.

\begin{lemma}\label{lem:var-sol-slicewise}
  Assume that $u_o\in L^1(0,T;W^{1,1}(\Omega))\cap C^0([0,T];L^2(\Omega))$.
   A function $u\in \Lw1(0,T;BV_{u_o}(\Omega))\cap C^0([0,T];L^2(\Omega))$,
   with $\partial_tu\in L^2(\Omega_T)$,
   is a variational solution of~\eqref{Cauchy-Dirichlet} in the sense of
   Definition~\ref{def:var-sol} if and only if
   \begin{equation}\label{var-sol-TVF-slicewise}
     \|Du(t)\|(\overline\Omega)
     \le
     \int_{\Omega\times\{t\}}\partial_tu(v-u)\dx
     +\|Dv\|(\overline\Omega)
   \end{equation}
  for a.e. $t\in(0,T)$ and every $v\in\BV_{u_o(t)}(\Omega)\cap
   L^2(\Omega)$ and if $u$ attains the initial values in the sense 
   \begin{equation}\label{initial values}
     u(0)=u_o(0)\qquad\mbox{in }L^2(\Omega).
   \end{equation}
 \end{lemma}

 \begin{remark}
   Condition~\eqref{var-sol-TVF-slicewise} can be reformulated in
   terms of the subdifferential
   \begin{equation*}
     \partial\Phi(u)=\big\{w\in L^2(\Omega)\colon 
     \Phi(u)+\langle w,v-u\rangle\le \Phi(v)\mbox{ for every $v\in L^2(\Omega)$}\big\}
   \end{equation*}
   of the functional $\Phi:L^2(\Omega)\to\R$, defined by 
   \begin{equation*}
     \Phi(u)=
       \begin{cases}
         \|D\bar u\|(\overline\Omega),
         &\mbox{if }
           \bar u\in\BV_{u_o}(\Omega)\cap L^2(\Omega^\ast),\\[2ex]
         \infty,&\mbox{if }\bar u\in L^2(\Omega^\ast)\setminus\BV_{u_o}(\Omega).
       \end{cases}
   \end{equation*}
   Here, the function $\bar u$ denotes the extension of $u$ by
   $u_o$ to $\Omega^\ast\setminus\Omega$. 
   By definition of the subdifferential,
   the variational inequality~\eqref{var-sol-TVF-slicewise} can be
   reformulated as
   \begin{equation*}
     -\partial_tu(t)\in \partial\Phi(u(t))
     \qquad\mbox{for a.e. }t\in(0,T).
   \end{equation*}
 \end{remark}
 
 \begin{proof}[Proof of Lemma~\ref{lem:var-sol-slicewise}]
   Assume that the map $u$ satisfies~\eqref{var-sol-TVF-slicewise} and~\eqref{initial values} and let
   $v\in\Lw1(0,T;\BV_{u_o}(\Omega))\cap
   C^0([0,T];L^2(\Omega))$ with $\partial_tv\in L^2(\Omega_T)$ be an
   arbitrary comparison function in~\eqref{var-sol-TVF}. 
   The function $v(t)$, for a.e. $t\in(0,T)$, is admissible in~\eqref{var-sol-TVF-slicewise}.  
   Integrating the resulting
   inequalities over time, we deduce 
   \begin{align*}
     \int_0^T\|Du(t)\|(\overline\Omega)\dt
     &\le
       \iint_{\Omega_T}\partial_tu(v-u)\dxt
       +\int_0^T\|Dv(t)\|(\overline\Omega)\dt.
   \end{align*}
   Since $\partial_tu,\partial_tv\in L^2(\Omega_T)$ and $u(0)=u_o(0)$ by~\eqref{initial values},
   an integration by parts implies
   \begin{align*}
     \iint_{\Omega_T}\partial_tu(v-u)\dxt
     &=
     \iint_{\Omega_T}\partial_tv(v-u)\dxt\\
     &\quad
        -\tfrac12\int_{\Omega}|v(T)-u(T)|^2\dx
        +\tfrac12\int_{\Omega}|v(0)-u_o(0)|^2\dx.
   \end{align*}
    Combining the two preceding formulae, we obtain~\eqref{var-sol-TVF},
    which proves that $u$ is a variational solution
    to~\eqref{Cauchy-Dirichlet}.

    \noindent
    For the opposite direction, we start with a variational solution
    $u\in \Lw1(0,T;\BV_{u_o}(\Omega))\cap  C^0([0,T];L^2(\Omega))$
    with $\partial_tu\in L^2(\Omega_T)$. 
   We begin with the observation that for any $v\in\BV(\Omega)\cap
   L^2(\Omega)$,
   the extension
   \begin{equation}\label{extension-v}
     \bar v(x,t)
     =
     \begin{cases}
        v(x),&x\in\Omega,\\[0.6ex]
        u_o(x,t),&x\in\Omega^\ast\setminus\Omega,
     \end{cases}
   \end{equation}
   defines a function $v\in\Lw1(0,T;\BV_{u_o}(\Omega))\cap
   C^0([0,T];L^2(\Omega^\ast))$ with 
   $\partial_tv=0$ in $\Omega_T$. This follows by applying
   Lemma~\ref{lem:ambrosio-et-al} separately on the time slices. 
   For a cut-off function in time 
   $\zeta\in C^\infty([0,T])$ with $\zeta(T)=0$
   and the extension $\bar v$ defined above,  we consider 
  \begin{equation*}
    w=u+\zeta(t)(\bar v-u).
  \end{equation*}
  We note that this function is
  admissible as comparison function in~\eqref{var-sol-TVF}, since the
  properties of $u$ and $\bar v$ imply 
  $w\in\Lw1(0,T;\BV_{u_o}(\Omega))\cap
  C^0([0,T];L^2(\Omega))$ and $\partial_tw\in L^2(\Omega_T)$.
 By convexity of the total variation functional
  \begin{equation*}
    \|Dw(t)\|(\overline\Omega)
    \le
    (1-\zeta(t))\|Du(t)\|(\overline\Omega)+\zeta(t)\|D\bar v(t)\|(\overline\Omega)   
  \end{equation*}
  for a.e. $t\in(0,T)$, we obtain
  \begin{align*}
     \int_0^T\zeta(t)\|Du(t)\|(\overline\Omega)\dt
     &\le
     \iint_{\Omega_T}\partial_t\big(u+\zeta(v-u)\big)(v-u)\zeta\dxt
       +\int_0^T\zeta(t)\|D\bar v(t)\|(\overline\Omega)\dt\\
     &\qquad+
       \tfrac12\int_\Omega \big|(1-\zeta(0))u(0)+\zeta(0)v-u_o(0)\big|^2\dx.
   \end{align*}
   Here, we also used the fact $\zeta(T)=0$, which ensures
   that no integral over the time slice at the final
   time occurs in the variational inequality. Integrating by
   parts, the integral involving
   the time derivative can be rewritten as 
   \begin{align*}
     &\iint_{\Omega_T}\partial_t\big(u+\zeta(v-u)\big)(v-u)\zeta\dxt\\
     &\qquad=
     \iint_{\Omega_T}\partial_tu (v-u)\zeta\dxt
     +
     \iint_{\Omega_T}\big(\zeta'\zeta|v-u|^2+\zeta^2
       \tfrac12\partial_t|v-u|^2\big)\dxt\\
     &\qquad=
       \iint_{\Omega_T}\partial_tu (v-u)\zeta\dxt
       -
       \tfrac12\int_\Omega\zeta^2(0)|v-u(0)|^2\dx.
   \end{align*}
   Combining the preceding formulae, we arrive at  
     \begin{align}\label{var-ineq-cutoff}
     &\int_0^T\zeta(t)\|Du(t)\|(\overline\Omega)\dt
     \le
       \iint_{\Omega_T}\partial_tu (v-u)\zeta\dxt
       +
       \int_0^T\zeta(t)\|D\bar v(t)\|(\overline\Omega)\dt\\\nonumber
     &\qquad\qquad+
       \tfrac12\int_\Omega \big|(1-\zeta(0))u(0)+\zeta(0)v-u_o(0)\big|^2\dx
       -
       \tfrac12\int_\Omega\zeta^2(0)|v-u(0)|^2\dx.
     \end{align}
     Our first goal is to show that the initial values are
     attained. To this end, we observe that
     an approximation argument implies the above inequality also for the
     characteristic function $\zeta=\chi_{[0,\tau]}$, for any
     $\tau\in(0,T)$. This gives
    \begin{align*}
     \int_0^\tau\|Du(t)\|(\overline\Omega)\dt
     &\le
       \iint_{\Omega_\tau}\partial_tu (v-u)\dxt
       +
       \int_0^\tau \|D\bar v(t)\|(\overline\Omega)\dt\\\nonumber
     &\qquad+
       \tfrac12\int_\Omega \big(|v-u_o(0)|^2-|v-u(0)|^2\big)\dx,
    \end{align*}
    for every $\tau\in(0,T)$. Since $\partial_tv=0$ in $\Omega_T$, an
    integration by parts gives
    \begin{align*}
      \iint_{\Omega_\tau}\partial_tu (v-u)\dxt
      &=
      -\tfrac12\iint_{\Omega_\tau}\partial_t|v-u|^2\dxt\\
      &=
      \tfrac12\int_\Omega |v-u(0)|^2\dx
      -\tfrac12\int_\Omega |v-u(\tau)|^2\dx.
    \end{align*}
    Combining the two preceding formulae, we arrive at
    \begin{align}\label{var-ineq-local}
      &\tfrac12\int_\Omega |v-u(\tau)|^2\dx
      +
      \int_0^\tau\|Du(t)\|(\overline\Omega)\dt\\\nonumber
      &\qquad\le
      \int_0^\tau \|D\bar v(t)\|(\overline\Omega)\dt
      +
      \tfrac12\int_\Omega |v-u_o(0)|^2\dx,
    \end{align}
    for any $v\in\BV(\Omega)\cap L^2(\Omega)$ and $\tau\in(0,T)$. 
    For a given $\eps>0$, we choose $u_{o,\eps}\in C^\infty_0(\Omega)$
    with $\|u_{o,\eps}-u_o(0)\|_{L^2(\Omega)}\le\eps$ and apply the
    preceding estimate with $v=u_{o,\eps}$. Discarding the
    second integral on the left-hand side of~\eqref{var-ineq-local},
    we have
    \begin{align*}
      \tfrac12\int_\Omega|u_{o,\eps}-u(\tau)|^2\dx
      \le
        \int_0^\tau\|D\bar u_{o,\eps}(t)\|(\overline\Omega)\dt
        +
        \tfrac12\int_\Omega|u_{o,\eps}-u_o(0)|^2\dx,
    \end{align*}
    which implies
    \begin{align}\label{L2-bound-initial}
      \tfrac14\int_\Omega|u_o(0)-u(\tau)|^2\dx
      &\le
        \tfrac12\int_\Omega|u_{o,\eps}-u(\tau)|^2\dx
        +
        \tfrac12\int_\Omega|u_{o,\eps}-u_o(0)|^2\dx\\\nonumber
      &\le
        \int_0^\tau\|D\bar u_{o,\eps}(t)\|(\overline\Omega)\dt
        +
        \int_\Omega|u_{o,\eps}-u_o(0)|^2\dx\\\nonumber
      &\le
        \int_0^\tau\|D\bar u_{o,\eps}(t)\|(\overline\Omega)\dt
        +
        \eps^2.
    \end{align}
    Using Lemma~\ref{lem:ambrosio-et-al}, we estimate the last
    integral by 
    \begin{equation*}
      \int_0^\tau \|D\bar u_{o,\eps}(t)\|(\overline\Omega)\dt
      \le
      \tau\int_\Omega|\nabla u_{o,\eps}|\dx
      +
      \int_0^\tau
      \int_{\partial\Omega}|T_{\R^n\setminus\overline\Omega}u_o(t)|\d\mathcal{H}^{n-1}\dt \to0
    \end{equation*}
    as $\tau\downarrow0$. 
    Letting first $\tau\downarrow0$ and then
    $\eps\downarrow0$ in~\eqref{L2-bound-initial}, we conclude that
    \begin{equation*}
      \int_\Omega|u_o(0)-u(0)|^2\dx
      =
      \lim_{\tau\downarrow0}\int_\Omega|u_o(0)-u(\tau)|^2\dx
      =0,
    \end{equation*}
    which implies the assertion $u(0)=u_o(0)$.

    It remains to show~\eqref{var-sol-TVF-slicewise}. For a cut-off
    function $\zeta\in C^\infty_0((0,T))$,
    inequality~\eqref{var-ineq-cutoff} and the fact that $u(0)=u_o(0)$ imply 
    \begin{equation*}
     \int_0^T\zeta(t)\|Du(t)\|(\overline\Omega)\dt
     \le
       \iint_{\Omega_T}\partial_tu (v-u)\zeta\dxt
       +
       \int_0^T\zeta(t)\|D\bar v(t)\|(\overline\Omega)\dt
     \end{equation*}
     for every $v\in\BV(\Omega)\cap L^2(\Omega)$, where $\bar v$ is
     defined by \eqref{extension-v}. 
     For a given time $s\in(0,T)$ and $0<\delta<\min\{s,T-s\}$,
     we use this estimate with the
    cut-off function $\zeta(t)=\frac1\delta\phi(\frac{s-t}\delta)$,
    where $\phi\in C^\infty_0((-1,1))$ denotes a standard mollifier.
    By letting $\delta\downarrow0$, we infer
    \begin{align*}
     \|Du(s)\|(\overline\Omega)
     &\le
       \int_{\Omega\times\{s\}}\partial_tu (v-u)\dx
       +
       \|D\bar v(s)\|(\overline\Omega)
   \end{align*}
    for a.e. $s\in(0,T)$ and every $v\in\BV(\Omega)\cap L^2(\Omega)$.
    This implies the remaining assertion~\eqref{var-sol-TVF-slicewise}
    and completes  the proof of Lemma~\ref{lem:var-sol-slicewise}.
 \end{proof}

\section{Equivalence of variational and weak solutions}

In this section we prove the equivalence of the two 
concepts of solution that have been introduced in 
Definition~\ref{def:Mazon-solution} and Definition~\ref{def:var-sol},
respectively. The precise statement of the result is the following.

\begin{theorem}\label{thm:main}
  A function $u\in\Lw1(0,T;\BV_{u_o}(\Omega))\cap C^0([0,T];L^2(\Omega))$
  with $\partial_tu\in L^2(\Omega_T)$
  is a variational solution of~\eqref{Cauchy-Dirichlet} if and only if
  it is a weak solution of~\eqref{Cauchy-Dirichlet}. 
\end{theorem}

In Subsection~\ref{sec:weakisvar} we show that a weak solution is a variational solution. 
The proof of the converse claim is presented in the remaining three subsections. The key step
is an elliptic duality result for the total variation functional in 
Subsection~\ref{sec:duality-TV}. This will be established as a stability
result by approximating the total variation by area-type functionals in Subection~\ref{sec:area-duality}.
Finally, in Subsection~\ref{sec:varisweak} we complete the proof of the claim that a variational solution is a weak solution. 

\subsection{Weak solutions are variational solutions}
\label{sec:weakisvar}

  Assume that $u\in\Lw1(0,T;\BV_{u_o}(\Omega))\cap C^0([0,T];L^2(\Omega))$,
  with $\partial_tu\in L^2(\Omega_T)$, is a weak solution according to Definition~\ref{def:Mazon-solution}. Let $z\in
  L^\infty(\Omega_T)$ with $\|z\|_{L^\infty}\le1$ be the vector field
  that is provided by Definition~\ref{def:Mazon-solution}. For
  a.e. $t\in(0,T)$ and any
  $v\in\BV_{u_o(t)}(\Omega)\cap L^2(\Omega)$, we use~\eqref{mazon-TV} and~\eqref{Anzellotti-bound} to
  deduce
  \begin{align*}
    \|Du(t)\|(\overline\Omega)
    +\int_{\Omega\times\{t\}}\partial_tu(u-v)\dx
    =
    (z(t),Dv)_{u_o}(\overline\Omega)
    \le
    \|Dv\|(\overline\Omega)
  \end{align*}
  for a.e. $t\in(0,T)$. This means that the variational
  inequality~\eqref{var-sol-TVF-slicewise} is satisfied on a.e. time
  slice, and Lemma~\ref{lem:var-sol-slicewise} implies
  that $u$ is a variational solution according to Definition~\ref{def:var-sol}. 

\subsection{An auxiliary result for the area functional}
\label{sec:area-duality}

The following approximation result for the area functional \eqref{def:areafcnl} will be applied in the proof of Theorem~\ref{thm:TV-duality}.

\begin{theorem}\label{thm:area-duality}
  Let $f\in W^{-1,\infty}(\Omega)\cap L^2(\Omega)$, $u_o\in
  W^{1,1}(\Omega)\cap L^2(\Omega)$, $\mu>0$, and $\lambda\in\R$ be
  given. Assume that
  $u\in\BV_{u_o}(\Omega)\cap L^2(\Omega^\ast)$ is a minimizer of the functional
  \begin{equation*}
    \Psi(v)=\Amu(v)+\int_{\Omega}\big(\tfrac\lambda2|v|^2+f(v-u_o)\big)\dx
  \end{equation*}
  in the space $\BV_{u_o}(\Omega)\cap L^2(\Omega^\ast)$.
  Then the vector field
  \begin{equation*}
    z=\frac{\nabla u}{\sqrt{\mu^2+|\nabla u|^2}}
  \end{equation*}
  satisfies

  \begin{equation}\label{div=f}
    \Div z=\lambda u+f \qquad\mbox{in $\Omega$}
  \end{equation}
  in the sense of distributions and we have the
  estimate
  \begin{align}\label{dual-problem-area}
    &\Amu(u)+\int_{\Omega}\big(\tfrac\lambda2|u|^2+f(u-u_o)\big)\dx\\\nonumber
    &\qquad\le
      \int_\Omega z\cdot\nabla u_o\dx
      +
      \tfrac\lambda2\int_\Omega|u_o|^2\dx
      +
      \mu\int_\Omega\sqrt{1-|z|^2}\dx.
  \end{align}
\end{theorem}

\begin{proof}
  For the proof of~\eqref{div=f}, we test the minimality of $u$
  with the comparison map $v_r=u-r\varphi\in\BV_{u_o}(\Omega)\cap L^2(\Omega^\ast)$,
  where $\varphi\in C^\infty_0(\Omega)$ and $r>0$. We apply the fact
  that $Du$ and $Dv_r$ have the same singular parts. This implies
  \begin{align*}
    &\int_{\Omega}\sqrt{\mu^2+|\nabla u|^2}\dx
    +\int_{\Omega}\big(\tfrac\lambda2|u|^2+f(u-u_o)\big)\dx\\
    &\qquad\le
    \int_{\Omega}\sqrt{\mu^2+|\nabla u-r\nabla\varphi|^2}\dx
    +
    \int_{\Omega}\big(\tfrac\lambda2|u-r\varphi|^2+f(u-r\varphi-u_o)\big)\dx,
  \end{align*}
  and, after dividing by $r>0$, 
    \begin{align*}
    \int_{\Omega}f\varphi\dx
    &\le
      \frac1r\int_{\Omega}\Big(\sqrt{\mu^2+|\nabla u-r\nabla\varphi|^2}-
      \sqrt{\mu^2+|\nabla u|^2}\Big)\dx\\
    &\quad+
      \frac\lambda2\frac1r\int_{\Omega}\big(|u-r\varphi|^2-|u|^2\big)\dx.
    \end{align*}
    Letting $r\downarrow0$ on the right-hand side, we deduce
    \begin{align*}
      \int_{\Omega}f\varphi\dx
      &\le
      \int_\Omega\frac{\partial}{\partial
      r}\Big|_{r=0}\sqrt{\mu^2+|\nabla u-r\nabla\varphi|^2}\dx
      +
      \frac\lambda2\int_\Omega\frac{\partial}{\partial
      r}\Big|_{r=0}\,|u-r\varphi|^2\dx\\
      &=-\int_\Omega\frac{\nabla
        u\cdot\nabla\varphi}{\sqrt{\mu^2+|\nabla u|^2}}\dx
        -
        \lambda\int_\Omega u\varphi\dx.
  \end{align*}
  Note that it is allowed to differentiate under the integrals because
  in the first case, 
  the derivative of the integral is dominated by $|\nabla\varphi|\in
  L^1(\Omega)$, and in the second integral, it is bounded by
  $2(|u|+|\varphi|)|\varphi|\in L^1(\Omega)$.
  In view of the definition of $z$, we have shown
  that 
  \begin{align*}
    \int_{\Omega}(\lambda u+f)\varphi\dx
    \le
    -\int_\Omega z\cdot \nabla\varphi\dx      
  \end{align*}
  holds true for every $\varphi\in C^\infty_0(\Omega)$. 
  Since the same estimate holds with $-\varphi$ instead of $\varphi$,
  the opposite inequality holds as well. This proves $\Div z=\lambda u+f$ in
  the distributional sense in $\Omega$.
   
  Next, we use $w_r=u+r(u_o-u)\in\BV_{u_o}(\Omega)\cap
  L^2(\Omega^\ast)$, for $r>0$, as a comparison function for the minimizer
  $u$.
  Since
  $D^sw_r=(1-r)D^su$, we obtain
  \begin{align*}
    &\int_{\Omega}\sqrt{\mu^2+|\nabla u|^2}\dx
    +\|D^su\|(\overline\Omega)
    +\int_{\Omega}\big(\tfrac\lambda2|u|^2+f(u-u_o)\big)\dx\\
    &\quad\le
      \int_{\Omega}\sqrt{\mu^2+|\nabla w_r|^2}\dx
      +(1-r)\|D^su\|(\overline\Omega)
      +\int_{\Omega}\big(\tfrac\lambda2|w_r|^2+(1-r)f(u-u_o)\big)\dx.
  \end{align*}
  Rearranging the terms and dividing by $r>0$, we deduce
  \begin{align*}
    &\|D^su\|(\overline\Omega)+\int_{\Omega}f(u-u_o)\dx\\
    &\qquad\le
      \frac1r\int_{\Omega}\big(\sqrt{\mu^2+|\nabla
      w_r|^2}-\sqrt{\mu^2+|\nabla u|^2}\big)\dx
      +
      \frac\lambda2\frac1r\int_{\Omega}\big(|w_r|^2-|u|^2\big)\dx.      
  \end{align*}
  Passing to the limit $r\downarrow0$, we have
  \begin{align*}
    &\|D^su\|(\overline\Omega)+\int_{\Omega}f(u-u_o)\dx\\
    &\qquad\le
      \int_{\Omega}\frac{\partial}{\partial r}\Big|_{r=0}\sqrt{\mu^2+|\nabla
      w_r|^2}\dx
      +
      \frac\lambda2\int_{\Omega}
      \frac{\partial}{\partial r}\Big|_{r=0}|w_r|^2\dx\\
    &\qquad=
      \int_{\Omega}\frac{\nabla u}{\sqrt{\mu^2+|\nabla
      u|^2}}\cdot(\nabla u_o-\nabla u)\dx
      +
      \lambda\int_\Omega u(u_o-u)\dx.
  \end{align*}
  Here, it is legitimate to differentiate under the integral because
  the derivative of the integrands are dominated by $|\nabla u_o-\nabla
  u|\in L^1(\Omega)$ and $2(|u|+|u_o|)|u_o-u|\in L^1(\Omega)$,
  respectively.
  By Young's inequality, 
  \begin{align*}
    \lambda\int_\Omega u(u_o-u)\dx
    \le
    \tfrac\lambda2\int_\Omega \big(|u_o|^2-|u|^2\big)\dx.
  \end{align*}
  Combining the two preceding estimates and
  recalling the definition of $z$, we arrive at
  \begin{align}\label{euler-area-1}
    &\int_{\Omega}\frac{|\nabla u|^2}{\sqrt{\mu^2+|\nabla
    u|^2}}\dx+\|D^su\|(\overline\Omega)+\int_{\Omega}\big(\tfrac\lambda2|u|^2+f(u-u_o)\big)\dx\\\nonumber
    &\qquad\qquad\le
    \int_{\Omega} z\cdot\nabla u_o\dx
    +\tfrac\lambda2\int_\Omega |u_o|^2\dx.
  \end{align}
  For the first integrand on the left-hand side, we have the identity
  \begin{align*}
    \frac{|\nabla u|^2}{\sqrt{\mu^2+|\nabla u|^2}}
    =
    \sqrt{\mu^2+|\nabla u|^2}-\mu\sqrt{1-|z|^2}.    
  \end{align*}
  This corresponds to the equality case in~\eqref{Fenchel-pointwise}. Thus~\eqref{euler-area-1} can be rewritten as
    \begin{align}\label{euler-area-2}
    &\int_{\Omega}\sqrt{\mu^2+|\nabla u|^2}\dx+\|D^su\|(\overline\Omega)
      +\int_{\Omega}\big(\tfrac\lambda2|u|^2+f(u-u_o)\big)\dx\\\nonumber
    &\qquad\le
    \int_{\Omega} z\cdot\nabla u_o\dx
    +
    \tfrac\lambda2\int_\Omega |u_o|^2\dx
    +
    \mu\int_\Omega\sqrt{1-|z|^2}\dx,
    \end{align}
    which is the asserted estimate \eqref{dual-problem-area}.
 \end{proof}

\subsection{A duality result for the total variation functional}
\label{sec:duality-TV}

The following duality result will be applied in the proof of
Theorem~\ref{thm:main}. More general duality results for problems with linear
growth have been established in \cite{Beck-Schmidt:2015}. Here, we give a
simple proof for a special case.
The argument applies an approximation process given by Theorem~\ref{thm:area-duality}.

\begin{theorem}\label{thm:TV-duality}
  Let $f\in W^{-1,\infty}(\Omega)\cap L^2(\Omega)$ with
  $\|f\|_{W^{-1,\infty}}\le 1$ and $u_o\in W^{1,1}(\Omega^\ast)\cap
    L^2(\Omega^\ast)$. Then we have 
  \begin{align}\label{dual-problem-TV}
    \inf_{u\in\BV_{u_o}(\Omega)\cap L^2(\Omega^\ast)}
    \bigg(\|Du\|(\overline\Omega)+\int_{\Omega}f(u-u_o)\dx\bigg)
    =
    \max_{z\in S^\infty_f(\Omega)}\int_\Omega z\cdot\nabla u_o\dx,
  \end{align}
  where $S^\infty_f(\Omega)=\{z\in
  L^\infty(\Omega,\R^n)\colon \|z\|_{L^\infty}\le1\ \mbox{and}\ \Div z=f\}$. 
\end{theorem}

\begin{proof}
Let  $u\in\BV_{u_o}(\Omega)\cap L^2(\Omega^\ast)$ and $z\in S^\infty_f(\Omega)$.
  By Lemma~\ref{lem:comparison-pairings} and~\eqref{Anzellotti-bound} we have
  \begin{align*}
    \int_\Omega z\cdot\nabla u_o\dx
    =
      (z,Du)_{u_o}(\overline\Omega)+\int_\Omega f(u-u_o)\dx
    \le
      \|Du\|(\overline\Omega)+\int_\Omega f(u-u_o)\dx.
  \end{align*}
  Taking the supremum on the left-hand side and the infimum on the
  right, we infer
  \begin{align}\label{dual-problem-easy-estimate}
    \sup_{z\in S^\infty_f(\Omega)}\int_\Omega z\cdot\nabla u_o\dx
    \le
    \inf_{u\in\BV_{u_o}(\Omega)\cap L^2(\Omega^\ast)}
    \bigg(\|Du\|(\overline\Omega)+\int_{\Omega}f(u-u_o)\dx\bigg).
  \end{align}

  In order to conclude the opposite inequality,
  we construct a minimizer of the functional  
  \begin{align*}
    \Psi_\mu(u):=\Amu(u)+\frac\mu2\int_\Omega |u|^2\dx+\int_\Omega f_\mu(u-u_o)\dx
  \end{align*}
  in the space $\BV_{u_o}(\Omega)\cap\L^2(\Omega^\ast)$ for
  $\mu\in(0,1)$, where $f_\mu:=(1-\mu)f$.
  Since every $u\in\BV_{u_o}(\Omega)\cap\L^2(\Omega^\ast)$ can be
  strictly approximated by functions $u_i\in (u_o+W^{1,1}_0(\Omega))\cap
  L^2(\Omega)$, $i\in\N$, in the sense of Lemma~\ref{lem:interior-approximation},
  we have 
  \begin{align*}
    \bigg|\int_\Omega f_\mu(u-u_o)\dx\bigg|
    &=
      \lim_{i\to\infty}\bigg|\int_\Omega f_\mu(u_i-u_o)\dx\bigg|\\[0.7ex]
    &\le
      \lim_{i\to\infty}\|f_\mu\|_{W^{-1,\infty}}\|\nabla u_i-\nabla u_o\|_{L^1(\Omega)}\\[0.7ex]
    &\le
    (1-\mu)\big(\|Du\|(\overline\Omega)+\|\nabla u_o\|_{L^1(\Omega)}\big)
  \end{align*}
  for every $u\in\BV_{u_o}(\Omega)\cap\L^2(\Omega^\ast)$. This implies the lower
  bound 
  \begin{align}
    \label{coercive}
    \Psi_\mu(u)
    &\ge
    \|Du\|(\overline\Omega)+\tfrac\mu2\|u\|_{L^2(\Omega)}^2-\bigg|\int_\Omega
      f_\mu(u-u_o)\dx\bigg|\\\nonumber
    &\ge
      \mu\|Du\|(\overline\Omega)
      +
      \tfrac\mu2\|u\|_{L^2(\Omega)}^2
      -
      (1-\mu)\|\nabla u_o\|_{L^1(\Omega)}.
  \end{align}
  We deduce that $\Psi_\mu$ is coercive on the space
  $\BV_{u_o}(\Omega)\cap L^2(\Omega^\ast)$. Since, moreover, $\Amu$ is
  convex, the direct method of the calculus of variations yields the
  existence of a minimizer $u_\mu\in\BV_{u_o}(\Omega)\cap
  L^2(\Omega^\ast)$ of $\Psi_\mu$. 
  We define the vector field
  \begin{equation*}
    z_\mu=\frac{\nabla u_\mu}{\sqrt{\mu^2+|\nabla u_\mu|^2}},
  \end{equation*}
  where $\nabla u_\mu$ denotes the Lebesgue density of the absolutely
  continuous part of $Du_\mu$.
  Theorem~\ref{thm:area-duality} with $\lambda=\mu$ and $f$ replaced
  by $f_\mu$ implies that
  \begin{equation}\label{Div-z-mu}
    \Div z_\mu =\mu u_\mu + f_\mu
    \qquad\mbox{in }\Omega
  \end{equation}
  in the sense of distributions, as well as the estimate 
  \begin{align*}
    &\Amu(u_\mu)
      +\int_{\Omega}\big(\tfrac\mu2|u_\mu|^2+f_\mu(u_\mu-u_o)\big)\dx\\\nonumber
    &\qquad\le
      \int_{\Omega} z_\mu\cdot\nabla u_o\dx
      +
      \tfrac\mu2\int_\Omega |u_o|^2\dx
      +\mu\int_\Omega\sqrt{1-|z_\mu|^2}\dx.
  \end{align*}
  Since
  \[
  \Amu(u_\mu)\ge\|Du_\mu\|(\overline\Omega)\ge(1-\mu)\|Du_\mu\|(\overline\Omega),
  \]
  the left-hand
  side can be bounded from below in terms of the infimum
  in~\eqref{dual-problem-TV}. More precisely, we have
  \begin{align}
    \label{upper-bound-inf-1}
    &(1-\mu)\inf_{u\in\BV_{u_o}(\Omega)\cap L^2(\Omega^\ast)}
    \bigg(\|Du\|(\overline\Omega)+\int_{\Omega}f(u-u_o)\dx\bigg)\\\nonumber
    &\qquad\le
      \int_{\Omega} z_\mu\cdot\nabla u_o\dx
      +
      \tfrac\mu2\int_\Omega |u_o|^2\dx
      +
      \mu\int_\Omega\sqrt{1-|z_\mu|^2}\dx
  \end{align}
  for every $\mu\in(0,1)$. Since $\|z_\mu\|_{L^\infty}\le1$ for
  every $\mu\in(0,1)$, we can find a sequence $\mu_i\downarrow0$ and a
  limit vector field $z_\ast\in L^\infty(\Omega,\R^n)$ with
  \begin{equation}\label{weak-star-z-mu}
    z_{\mu_i}\wsto z_\ast
    \qquad\mbox{weakly$\ast$ in $L^\infty(\Omega,\R^n)$, as $i\to\infty$.}
  \end{equation}
  Using estimate~\eqref{coercive} and the minimality of $u_\mu$, we
  infer the bound
  \begin{align*}
    \tfrac{\mu_i}2\|u_{\mu_i}\|_{L^2(\Omega)}^2
    &\le
    \Psi_{\mu_i}(u_{\mu_i})+(1-\mu_i)\|\nabla u_o\|_{L^1(\Omega)}\\
    &\le
    \Psi_{\mu_i}(u_o)+\|\nabla u_o\|_{L^1(\Omega)}\\
    &\le
    \int_\Omega\Big(\sqrt{1+|\nabla u_o|^2}+|u_o|^2+|\nabla u_o|\Big)\dx.
  \end{align*}
  This implies that the sequence of functions $\sqrt{\mu_i}u_{\mu_i}$, $i\in\N$, is bounded
  in $L^2(\Omega)$, and we get
  \begin{equation*}
    \mu_i\int_\Omega u_{\mu_i}\varphi\dx
    \le \mu_i\|u_{\mu_i}\|_{L^2(\Omega)}\|\varphi\|_{L^2(\Omega)}
    \to0
  \end{equation*}
  as $i\to\infty$, for every $\varphi\in
  C^\infty_0(\Omega)$.
  We use this together with the convergence~\eqref{weak-star-z-mu} to
  pass to the limit in~\eqref{Div-z-mu}, which implies 
  $\Div z_\ast=f$ in $\Omega$, in the sense of distributions. Since
  \[
  \|z_\ast\|_{L^\infty}\le\liminf_{i\to\infty}\|z_{\mu_i}\|_{L^\infty}\le1,
  \]
  we infer $z_\ast\in S^\infty_f(\Omega)$. 
  Next, we use the convergence~\eqref{weak-star-z-mu}
  to pass to the limit $i\to\infty$
  in~\eqref{upper-bound-inf-1} and arrive at 
  \begin{align*}
    &\inf_{u\in\BV_{u_o}(\Omega)\cap L^2(\Omega^\ast)}
    \bigg(\|Du\|(\overline\Omega)+\int_{\Omega}f(u-u_o)\dx\bigg)\\
    &\qquad\qquad\le
    \int_\Omega z_\ast\cdot\nabla u_o\dx
    \le
    \sup_{z\in S^\infty_f(\Omega)}\int_\Omega z\cdot\nabla u_o\dx\\
    &\qquad\qquad\le
    \inf_{u\in\BV_{u_o}(\Omega)\cap L^2(\Omega)}
    \bigg(\|Du\|(\overline\Omega)+\int_{\Omega}f(u-u_o)\dx\bigg).
  \end{align*}
  For the last inequality, we recall~\eqref{dual-problem-easy-estimate}.
  We conclude that we have an equality throughout, and in
  particular, the supremum above is attained. This completes the proof
  of~\eqref{dual-problem-TV}.
\end{proof}

\subsection{Variational solutions are weak solutions}
\label{sec:varisweak}

In this subsection, we complete the proof of Theorem~\ref{thm:main}. 
  To this end, assume that 
  $u\in \Lw1(0,T;\BV_{u_o}(\Omega))\cap
  C^0([0,T];L^2(\Omega))$ is a variational solution
  of~\eqref{Cauchy-Dirichlet}
  with $\partial_tu\in L^2(\Omega_T)$.
  Lemma~\ref{lem:var-sol-slicewise} implies that the initial values
  are attained in the sense $u(0)=u_o(0)$ and
  that the slicewise variational inequality  
   \begin{equation}\label{var-ineq-TV-slicewise}
     \|Du(t)\|(\overline\Omega)
     \le
     \int_{\Omega\times\{t\}}\partial_tu(v-u)\dx
     +\|Dv\|(\overline\Omega)
   \end{equation}
   holds true for every $v\in\BV_{u_o(t)}(\Omega)\cap
   L^2(\Omega^\ast)$ and a.e. $t\in(0,T)$. 
   For the proof of~\eqref{mazon-TV}, we begin with the observation that the variational inequality~\eqref{var-ineq-TV-slicewise}
   implies $\partial_tu\in \Lw\infty(0,T; W^{-1,\infty}(\Omega))$ with 
   \begin{equation}\label{dtu-dual-est}
     \esssup_{t\in(0,T)}\|\partial_tu(t)\|_{W^{-1,\infty}(\Omega)}\le 1.
   \end{equation}
   In order to prove this claim, we consider an arbitrary $\varphi\in
   L^1(0,T;W^{1,1}_0(\Omega))$ and use $v=u(t)-\varphi(t)$
   as a comparison function in
   the variational inequality~\eqref{var-ineq-TV-slicewise}. After integrating
   over $t\in(0,T)$, we obtain the bound 
   \begin{align*}
     \iint_{\Omega_T}\partial_tu\varphi\dxt
     &\le
     \int_0^T\Big(\big\|Du(t)-D\varphi(t)\big\|(\overline\Omega)-\|Du(t)\|(\overline\Omega)\Big)\dt\\
     &\le
     \iint_{\Omega_T}|\nabla\varphi|\dxt.
   \end{align*}
   The preceding estimate implies 
   $$
     \partial_tu\in \big[L^1(0,T;W^{1,1}_0(\Omega))\big]'=
     \Lw\infty(0,T;W^{-1,\infty}(\Omega)),
   $$
   together with~\eqref{dtu-dual-est}. 
   Next, we note that the variational inequality
   \eqref{var-ineq-TV-slicewise} corresponds to a minimization
   property of 
   $u(t)\in\BV_{u_o(t)}(\Omega)\cap
   L^2(\Omega^\ast)$. More precisely, for a.e. $t\in(0,T)$, the
   function $u(t)$ is a minimizer of the functional
   \begin{equation*}
     \Psi(v)=\|Dv\|(\overline\Omega)+\int_{\Omega\times\{t\}}\partial_tu(v-u_o)\dx
   \end{equation*}
   in the space $\BV_{u_o(t)}(\Omega)\cap L^2(\Omega^\ast)$.
   In view of~\eqref{dtu-dual-est}, Theorem~\ref{thm:TV-duality}
   is applicable with the choice
   $f=\partial_tu(t)\in W^{-1,\infty}(\Omega)\cap L^2(\Omega)$, for
   a.e. $t\in(0,T)$.  
   Theorem~\ref{thm:TV-duality} implies
  \begin{equation}\label{duality-slicewise}
    \|Du(t)\|(\overline\Omega)
      +\int_{\Omega\times\{t\}}\partial_tu(u-u_o)\dx
    =
    \max_{\tilde z\in S^\infty_{\partial_tu(t)}(\Omega)}\,
     \int_{\Omega\times\{t\}} \tilde z\cdot\nabla u_o\dx
  \end{equation}
  for a.e. $t\in(0,T)$. Our next goal is to show that there is a
  vector field $z\in L^\infty(\Omega_T,\R^n)$ such that $z(t)\in
  L^\infty(\Omega,\R^n)$ 
  realizes the maximum in~\eqref{duality-slicewise} for a.e. $t\in(0,T)$. 
  By definition of the maximum, for a.e. $t\in(0,T)$ we can choose    
  a vector field $z_\ast(t)\in L^\infty(\Omega,\R^n)$ with
   $\|z_\ast(t)\|_{L^\infty(\Omega)}\le1$ and 
   $\Div z_\ast(t)=\partial_tu(t)$ in $\Omega$, such that 
   \begin{equation}
     \label{choice-z}
     \int_{\Omega\times\{t\}} z_\ast\cdot\nabla u_o\dx
     =
     \max_{\tilde z\in S^\infty_{\partial_tu(t)}(\Omega)}\,
     \int_{\Omega\times\{t\}} \tilde z\cdot\nabla u_o\dx.
   \end{equation}
   However, at this stage we can not rule out the possibility that
   $t\mapsto z_\ast(t)$ is not measurable.
   In this case, we need to replace $z_\ast$
   by a measurable vector field $z\in L^\infty(\Omega_T,\R^n)$. To
   show the existence of such a vector field, 
   we identify elements in $L^\infty(\Omega_T,\R^n)$ with bounded linear
   functionals on $L^1(\Omega_T,\R^n)$.
   Let us consider the subspace
   \begin{equation*}
     \mathcal{W}=\big\{ r\nabla u_o+\nabla\varphi\ \colon\,
     r\in\R,\ 
     \varphi\in C^\infty_0(\Omega_T)\big\}\subset L^1(\Omega_T,\R^n).
   \end{equation*}
   For a given element $V\in\mathcal{W}$, we choose $r\in\R$ and
   $\varphi\in C^\infty_0(\Omega_T)$ with $V=r\nabla
   u_o+\nabla\varphi$. For a.e. $t\in(0,T)$, we have
   \begin{align*}
     \int_{\Omega\times\{t\}} z_\ast\cdot V\dx
     &=
       r\int_{\Omega\times\{t\}} z_\ast\cdot \nabla u_o\dx
       +
       \int_{\Omega\times\{t\}} z_\ast\cdot \nabla\varphi\dx
       \\
     &=
       r\max_{\tilde z\in S^\infty_{\partial_tu(t)}(\Omega)}\,
       \int_{\Omega\times\{t\}} \tilde z\cdot\nabla u_o\dx
       -
       \int_{\Omega\times\{t\}} \partial_tu\,\varphi\dx.
   \end{align*}
   For the last identity, we used the choice of $z_\ast(t)$ according
   to \eqref{choice-z} and the fact that $\Div z_\ast(t)=\partial_tu(t)$.
   We observe that the maximum in the last line
   depends measurably on time because it coincides with a measurable
   function by~\eqref{duality-slicewise},
   cf. Lemma~\ref{lem:TV-measurable}.
   We conclude that the
   left-hand side of the preceding formula depends measurably on time
   as well. Moreover, because of $\|z_\ast(t)\|_{L^\infty}\le 1$
   for a.e. $t\in(0,T)$, we have
   \begin{equation}\label{bound-z-star-slicewise}
     \bigg|\int_{\Omega\times\{t\}} z_\ast\cdot V\dx\bigg|
     \le
     \int_{\Omega\times\{t\}}|V|\dx
   \end{equation}
   for a.e. $t\in(0,T)$. 
   Therefore, we may define a linear functional 
   \begin{equation*}
     \ell: \mathcal{W}\to\R,
     \quad
     V\mapsto \int_0^T\bigg(\int_{\Omega\times\{t\}}z_\ast\cdot V\dx\bigg)\dt.
   \end{equation*}
   Integrating estimate~\eqref{bound-z-star-slicewise} over time, 
   we infer the bound
   \begin{equation*}
     |\ell(V)|\le \|V\|_{L^1(\Omega_T)}
     \quad\mbox{for every $V\in\mathcal{W}$.}
   \end{equation*}
   This means that $\ell$ is a bounded linear functional on
   $\mathcal{W}$ with
   $\|\ell\|_{\mathcal{W}'}\le1$. By the Hahn-Banach theorem, there
   exists an extension $L\in [L^1(\Omega_T,\R^n)]'$ with
   $L|_{\mathcal{W}}=\ell$ and
   \begin{equation*}
     \|L\|_{[L^1(\Omega_T,\R^n)]'}=\|\ell\|_{\mathcal{W}'}\le 1.
   \end{equation*}
   The Riesz representation theorem yields a vector field $z\in
   L^\infty(\Omega_T,\R^n)$ with
   \begin{equation}\label{bound-z-one}
     \|z\|_{L^\infty(\Omega_T)}=\|L\|_{[L^1(\Omega,\R^n)]'}\le1
   \end{equation}
    and 
   \begin{equation}\label{Riesz}
     \iint_{\Omega_T}z\cdot V\dxt
     =
     L(V)
     =
     \int_0^T\bigg(\int_{\Omega\times\{t\}}z_\ast\cdot V\dx\bigg)\dt
   \end{equation}
   for every $V\in\mathcal{W}$. We exploit this identity in two
   ways. First, we choose $V=\nabla\varphi\in\mathcal{W}$, where
   $\varphi\in C^\infty_0(\Omega_T)$, and deduce
   \begin{equation*}
     \iint_{\Omega_T}z\cdot\nabla\varphi\dxt
     =
     \int_0^T\bigg(\int_{\Omega\times\{t\}}z_\ast\cdot \nabla\varphi\dx\bigg)\dt
     =
     -\iint_{\Omega_T}\partial_tu \,\varphi\dxt,
   \end{equation*}
   which means $\Div z=\partial_tu$ in $\Omega_T$, in the sense of
   distributions. In view of~\eqref{bound-z-one}, we infer $z(t)\in
   S^\infty_{\partial_tu(t)}(\Omega)$ for a.e. $t\in(0,T)$.
   Second, the choice $V=\nabla u_o\in\mathcal{W}$
   in~\eqref{Riesz} implies 
   \begin{align*}
     \iint_{\Omega_T}z\cdot \nabla u_o\dxt
     &=
       \int_0^T\bigg(\int_{\Omega\times\{t\}}z_\ast\cdot \nabla
       u_o\dx\bigg)\dt\\
     &=
       \int_0^T\bigg(\max_{\tilde z\in S^\infty_{\partial_tu(t)}(\Omega)}\,
       \int_{\Omega\times\{t\}} \tilde z\cdot\nabla u_o\dx\bigg)\dt\\
     &\ge
       \iint_{\Omega_T}z\cdot \nabla u_o\dxt,
   \end{align*}
   where we used~\eqref{choice-z} and the fact $z(t)\in
   S^\infty_{\partial_tu(t)}(\Omega)$ for a.e. $t\in(0,T)$. 
   We deduce that the last inequality must be an identity, which implies
   \begin{align*}
     \int_{\Omega\times\{t\}}z\cdot \nabla u_o\dx
     =
     \max_{\tilde z\in S^\infty_{\partial_tu(t)}(\Omega)}\,
       \int_{\Omega\times\{t\}} \tilde z\cdot\nabla u_o\dx
   \end{align*}
   for a.e. $t\in(0,T)$. Hence, we have found the desired vector field
   $z\in L^\infty(\Omega_T,\R^n)$ such that $z(t)\in
   S^\infty_{\partial_tu(t)}(\Omega)$ realizes the
   maximum in~\eqref{duality-slicewise} for
   a.e. $t\in(0,T)$. Equation~\eqref{duality-slicewise} implies the identity
   \begin{align*}
     \|Du(t)\|(\overline\Omega)
     +\int_{\Omega\times\{t\}}\partial_tu(u-u_o)\dx
     =
     \int_{\Omega\times\{t\}} z\cdot\nabla u_o\dx
   \end{align*}
   for a.e. $t\in(0,T)$.
   For the proof of~\eqref{mazon-TV},
   it remains to replace $u_o$ by an arbitrary 
   $v\in\BV_{u_o(t)}(\Omega)\cap L^2(\Omega^\ast)$ in the preceding identity.
   This can be done with the help of
   Lemma~\ref{lem:comparison-pairings}. Since
   $\Div z(t)=\partial_tu(t)$, Lemma~\ref{lem:comparison-pairings} implies 
  \begin{equation*}
    \int_{\Omega\times\{t\}} z\cdot\nabla u_o\dx
    =
    (z(t),Dv)_{u_o}(\overline\Omega)
    +
    \int_{\Omega\times\{t\}} \partial_tu(v-u_o)\dx.
  \end{equation*}
  Combining the two preceding identities, we arrive at 
  \begin{align*}
    \|Du(t)\|(\overline\Omega)
    +\int_{\Omega\times\{t\}}\partial_tu(u-v)\dx
    =
    (z(t),Dv)_{u_o}(\overline\Omega)  
  \end{align*}
  for a.e. $t\in(0,T)$ and every 
  $v\in\BV_{u_o(t)}(\Omega)\cap L^2(\Omega^\ast)$, which is
  the assertion~\eqref{mazon-TV}.
  Therefore, the function $u$ is a weak solution
  of~\eqref{Cauchy-Dirichlet} in the sense of Definition~\ref{def:Mazon-solution}. This completes
  the proof of Theorem~\ref{thm:main}.

\noindent
Juha Kinnunen,
Aalto University, Department of Mathematics,
P.O. Box 11100, FI-00076 Aalto, Finland.
Email: juha.k.kinnunen@aalto.fi

\medskip\noindent
Christoph Scheven, Fakult\"at f\"ur Mathematik, 
Universit\"at Duisburg-Essen, 45117 Essen, Germany.
Email: christoph.scheven@uni-due.de

\medskip\noindent
KEY WORDS AND PHRASES: Parabolic variational integral, total variation flow, Cauchy-Dirichlet problem.

\medskip\noindent
2010 MATHEMATICS SUBJECT CLASSIFICATION: 35K67, 35K20, 35K92.

\end{document}